\documentclass{amsart}
\usepackage{amscd}
\usepackage[T1]{fontenc}
\usepackage[utf8]{inputenc}
\usepackage{lmodern}
\usepackage{enumerate}
\def\car{{\mathbf 1}}
\def\L{{\mathcal L}}
\def\half{1/2}
\def\/{\, | \,}
\def\n{\nabla}
\def\cn{\check{\nabla}}
\def\cu{\check{u}}

\def\ee{\epsilon}
\def\div{\delta}

\def\ee{\varepsilon}
\def\<{\langle}
\def\>{\rangle}
\def\n{\nabla}
\def\div{\delta}
\def\G{{\mathfrak G}}
\def\P{{\mathbf P}}

\def\H{{\mathbf  H}}

\def\WW{{\text{\tiny W}}}
\def\I{{\mathcal I}}
\def\K{{\mathcal K}}
\def\F{{\mathcal F}}

\def\R{{\mathbf R}}\def\N{{\mathbf N}}

\def\VV{\check{V}_T}
\def\CC{{\mathcal C}_0}
\def\LL{{\mathbb L}}
\def\comega{\check{\omega}}
\def\cdelta{\check{\delta}}
\def\d{\text{ d}}
\def\L{{\mathcal L}}
\def\({\Bigl(}
\def\){\Bigr)}
\def\E{{E^0}}
\def\cE{{\check{E}^0}}

\def\taut{\tau_{\scriptscriptstyle{T}}}

\newcommand{\esp}[1]{{\mathbf E}\left[{#1}\right]}
\def\D{{\mathbb D}}
\newcommand{\trace}{\operatorname{trace}}
\newcommand{\Dom}{\operatorname{Dom}}
\newcommand{\Hol}{\operatorname{Hol}}
\newcommand{\Id}{\operatorname{Id}}

\newcommand{\id}{\operatorname{Id}}
\newcommand{\RE}{{\mathbf R}}

\newtheorem{theorem}{Theorem}[section]
\newtheorem{prop}[theorem]{Proposition}%
\newtheorem{corollary}[theorem]{Corollary}
\newtheorem{lemma}[theorem]{Lemma}
\newtheorem{defn}{Definition}[section]
\newtheorem{hyp}{Hypothesis}

\newtheorem{rem}{Remark}[section]
\newtheorem{example}{Example}

\begin{document}
\title{Time reversal of Volterra processes driven stochastic
  differential equations} \author{L. Decreusefond} \address{Institut
  TELECOM - TELECOM ParisTech - CNRS LTCI\\ Paris, France}
\email{Laurent.Decreusefond@telecom-paristech.fr} \date{}
\keywords{fractional Brownian motion, Malliavin calculus, stochastic
  differential equations, time reversal.}  \subjclass[2010]{60G18,
  60H10, 60H07, 60H05}
\thanks{The author would like to thank the anonymous referees for
  their thorough reviews, their comments and
  suggestions significatively contributed to improve the quality
  of this contribution.}

\begin{abstract}
  We consider stochastic differential equations driven by some
  Volterra processes. Under time reversal, these equations are
  transformed into past dependent stochastic differential equations
  driven by a standard Brownian motion. We are then in position to
  derive existence and uniqueness of solutions of the Volterra driven
  SDE considered at the beginning.
\end{abstract}
\maketitle{}
\markboth{Time reversal of fBm driven SDEs}{L. Decreusefond}
\section{Introduction}
\label{sec_introduction}
Fractional Brownian motion (fBm for short) of Hurst index $H\in [0,\, 1]$ is the Gaussian process which
admits the following representation: For any $t\ge 0$,
\begin{equation*}
  B^H(t)=\int_0^t K_H(t,\, s)\d B(s)
\end{equation*}
where $B$ is a one dimensional Brownian motion and $K_H$ is a
triangular kernel, i.e. $K_H(t,s)=0$ for $s>t$, the definition of
which is given in \eqref{defdekh}.
Fractional Brownian motion  is probably the first  process which
is not a semi-martingale and for which it is still interesting  to develop a stochastic
calculus. That means we want to define a stochastic integral and solve
stochastic differential equations driven by such a process. From the
very beginning of this program, two approaches do exist. One approach
is based on the   Hölder
continuity or the finite $p$-variation of the fBm sample-paths. The other way to proceed
relies on the gaussiannity of fBm.  The former is mainly deterministic
and was initiated by Zähle \cite{zahle98}, Feyel, de la Pradelle
\cite{feyel96} and Russo, Vallois \cite{russo95,russo96}.  Then, came
the notion of rough paths introduced by Lyons \cite{lyons98}, whose
application to fBm relies on the work of Coutin,
Qian\cite{MR1883719}. These works have been extended in the subsequent
works
\cite{MR2319719,DN:flots-06,MR2144230,MR2257130,MR2144234,MR2228694,MR2348055,MR2359060,MR2138713,MR2268434,MR2236631}.
A new way of thinking came with the independent but related works of
Feyel, de la Pradelle \cite{EJP2006-35} and Gubinelli
\cite{MR2091358}. The integral with respect to fBm was shown to exist
as the unique process satisfying some characterization (analytic in
the case of \cite{EJP2006-35}, algebraic in \cite{MR2091358}). As a
byproduct, this showed that almost all the existing integrals
throughout the literature were all the same as they all satisfy these
two conditions. Behind each approach but the last too, is a
construction of an integral defined for a regularization of fBm, then
the whole work is to show that under some convenient hypothesis, the
approximate integrals converge to a quantity which is called the
stochastic integral with respect to fBm. The main tool to prove the
convergence is either integration by parts in the sense of fractional
deterministic calculus, either enrichment of the fBm by some iterated
integrals proved to exist independently or by analytic continuation
\cite{Unterberger:2009rt,Unterberger:2009yq}.

In the probabilistic approach
\cite{MR2000m:60059,alos00,decreusefond02,MR1956051,decreusefond03.1,decreusefond96_2,MR2397797,MR1893308,MR2493996},
the idea is also to define an approximate integral and then prove its
convergence. It turns out that the key tool is here the integration by
parts in the sense of Malliavin calculus.

In dimension greater than one, with the deterministic approach, one
knows how to define the stochastic integral and prove existence and
uniqueness of fBm driven SDEs for fBm with Hurst index greater than
$1/4$. Within the probabilistic framework, one knows how to define a
stochastic integral for any value of $H$ but one cannot prove existence
and uniqueness of SDEs whatever the value of $H$. The primary
motivation of this work is to circumvent this problem.

In \cite{decreusefond03.1,decreusefond96_2}, we defined stochastic
integrals with respect to fBm as a ``damped-Stratonovitch'' integral
with respect to the underlying standard Brownian motion. This integral
is defined as the limit of Riemann-Stratonovitch sums, the convergence
of which is proved after an integration by parts in the sense of
Malliavin calculus. Unfortunately, this manipulation generates
non-adaptiveness: Formally the result can be expressed as
\begin{equation*}
  \int_0^t u(s)\circ \d B^H(s)=\delta(\K^*_t u)+\trace(\K^*_t \n u),
\end{equation*}
where $\K$ is defined by 
\begin{equation*}
  \K f(t)=\frac{d}{dt}\int_0^t K_H(t,\,s)f(s) \d s
\end{equation*}
and $\K^*_t$ is the adjoint of $\K$ in $\L^2([0,\, t],\, \R)$. In
particular, there exists $k$ such that 
\begin{equation*}
  \K^*_tf(s)=\int_s^t k(t,u) f(u)\d u
\end{equation*}
for any $f\in \L^2([0,\, t],\, \R)$ so that even if $u$ is adapted (with respect to the Brownian filtration), the
process $(s\mapsto \K^*_tu(s))$ is anticipative. However, the
stochastic integral process $(t\mapsto \int_0^t u(s)\circ \d B^H(s))$
remains adapted, hence, the anticipative aspect is, in some sense, artificial.
The motivation of this work is to show, that up to time reversal, we
can work with adapted process and Itô integrals.  The
time-reversal properties of fBm were already studied in
\cite{MR2346508} in a different context: It was shown there that the
time-reversal of the solution of an fBm-driven SDE of the form 
\begin{equation*}
  dY(t)=u(Y(t))\d t + \d B^H(t)
\end{equation*}
is still a process  of the same form. With a slight adaptation of our
method to fBm-driven SDEs with drift, one  should recover the main theorem
of \cite{MR2346508}.

 In what follows,
there is no restriction about the dimension but we need to assume that
 any component of $B^H$ is an fBm of  Hurst index  greater
than $1/2.$
Consider that we want to solve the equation \begin{equation}
  \label{eq:1}
  X_t = x + \int_0^t \sigma(X_s)\circ \d  B^H(s), \, 0\le t \le T
\end{equation}
where $\sigma$ is a deterministic function whose properties will be
fixed below. It turns out that it is essential to investigate the more
general equations:
\begin{equation}
  \label{eq:A}
  \tag{A}
  X_{r,\, t}=x+\int_r^t \sigma(X_{r,\, s})\circ \d B^H(s), \ 0\le r\le
  t\le T.
\end{equation}
The strategy is then the following: We will first consider the
reciprocal problem:
\begin{equation}
  \label{eq:B}
  \tag{B}
  Y_{r,\, t} =x-\int_r^t \sigma(Y_{s,\, t})\circ \d B^H(s),  \ 0\le r\le
  t\le T.
\end{equation}

The first critical point is that when we consider $\{Z_{r,\,
  t}:=Y_{t-r,\, t}, \, r\in [0,\, t]\},$ this process solves an
adapted, past dependent, stochastic differential equation with respect
to a standard Brownian motion. Moreover, because $K_H$ is
lower-triangular and sufficiently regular, the trace term vanishes in
the equation defining $Z$. We have then reduced the problem to an SDE
with coefficients dependent on the past, a problem which can be
handled by the usual contraction methods. We do not claim that the
results presented are new (for instance see the brilliant monograph \cite{MR2604669}
for detailed results obtained via rough paths theory) but it seems
interesting to have purely probabilistic methods which show that fBm
driven SDEs do have strong solutions which are
homeomorphisms. Moreover, the approach given here shows the
irreducible difference between the case $H<1/2$ and $H>1/2$ : The
trace term only vanishes in the latter situation, so that such an SDE is
merely a usual SDE with past-dependent coefficients. This
representation may be fruitful for instance, to analyze the  support and prove
the  absolute continuity of solutions of \eqref{eq:1}.

This paper is organized as
follows: After some preliminaries on fractional Sobolev spaces, often
called Besov-Liouville space, we address, in Section \ref{sec_stoch-integr}, the problem of
Malliavin calculus and time reversal. This part is interesting in its
own since stochastic calculus of variations is a framework oblivious
to time. Constructing such a notion of time is achieved using the
notion of resolution of the identity as introduced in
\cite{ustunel95_1,MR1071539}. We then introduce the second key ingredient which
is the notion of strict causality or quasi-nilpotence, see
\cite{zakai92} for a related application. In Section~\ref{sec_time-reversal}, we show that solving Equation
\eqref{eq:B} reduces to solve a past dependent stochastic differential
equation with respect to a standard Brownian motion, see Equation
\eqref{eq:C} below. Then, we prove
existence, uniqueness and some properties of this equation.  Technical
lemmas are postponed to Section \ref{sec:proofs}.

\section{Besov-Liouville Spaces}
\label{sec_preliminaries}


Let $T>0$ be fix real number. For a measurable function $f\, :\, [0,\,
T]\to \R^n$, we define $\taut f$ by \begin{equation*} \taut
  f(s)=f(T-s) \text{ for any } s\in [0,\, T].
\end{equation*}
For $t\in [0,\, T]$, $e_tf$ will represent the restriction of $f$ to
$[0,\, t]$, i.e., $e_tf=f\car_{[0,\, t]}.$ For any linear map $A,$ we
denote by $A^*_{T},$ its adjoint in $\L^2([0,T];\, \R^n).$
For $\eta\in (0,1],$ the space of $\eta$-Hölder continuous functions
on $[0,\, T]$ is equipped with the norm
\begin{equation*}
  \| f\|_{\Hol(\eta)}=\sup_{0<s<t<T}\frac{|f(t)-f(s)|}{|t-s|^\eta}+\| f\|_\infty.
\end{equation*}
Its topological dual is denoted by $\Hol(\eta)^*.$ For $f\in
\L^1([0,T];\, \R^n;\d t),$ (denoted by $\L^1$ for short) the left and
right fractional integrals of $f$ are defined by~:
\begin{align*}
  (I_{0^+}^{\gamma}f)(x) & =
  \frac{1}{\Gamma(\gamma)}\int_0^xf(t)(x-t)^{\gamma-1}\d t\ ,\ x\ge
  0,\\
  (I_{T^-}^{\gamma}f)(x) & =
  \frac{1}{\Gamma(\gamma)}\int_x^Tf(t)(t-x)^{\gamma-1}\d t\ ,\ x\le T,
\end{align*} where $\gamma>0$ and $I^0_{0^+}=I^0_{T^-}=\Id.$ For any
$\gamma\ge 0$, $p,q\ge 1,$ any $f\in \L^p$ and $g\in \L^q$ where
$p^{-1}+q^{-1}\le \gamma$, we have~: \begin{equation}
  \label{int_parties_frac}
  \int_0^T f(s)(I_{0^+}^\gamma g)(s)\d s = \int_0^T (I_{T^-}^\gamma 
  f)(s)g(s)\d s.
\end{equation}
The Besov-Liouville space $I^\gamma_{0^+}(\L^p):= \I_{\gamma,p}^+$ is
usually equipped with the norm~:
\begin{equation}
  \label{normedansIap}
  \|  I^{\gamma}_{0^+}f \| _{ \I_{\gamma,p}^+}=\| f\|_{\L^p}.
\end{equation}
Analogously, the Besov-Liouville space $I^\gamma_{T^-}(\L^p):=
\I_{\gamma,p}^-$ is usually equipped with the norm~:
\begin{equation*}
  \| I^{-\gamma}_{T^-}f \| _{ \I_{\gamma,p}^-}=\|  f\|_{\L^p}.
\end{equation*}
We then have the following continuity results (see
\cite{feyel96,samko93})~:
\begin{prop}
  \label{prop:proprietes_int_RL}
  \begin{enumerate}[i.]
  \item \label{inclusionLpLq} If $0<\gamma <1,$ $1< p <1/\gamma,$ then
    $I^\gamma_{0^+}$ is a bounded operator from $\L^p$ into $\L^q$
    with $q=p(1-\gamma p)^{-1}.$
  \item For any $0< \gamma <1$ and any $p\ge 1,$ $\I_{\gamma,p}^+$ is
    continuously embedded in $\Hol(\gamma- 1/p)$ provided that
    $\gamma-1/p>0.$
  \item For any $0< \gamma< \beta <1,$ $\Hol (\beta)$ is compactly
    embedded in $\I_{\gamma,\infty}.$
  \item For $\gamma p<1,$ the spaces $\I_{\gamma,p}^+$ and
    $\I_{\gamma,p}^{-}$ are canonically isomorphic. We will thus use
    the notation $\I_{\gamma,p}$ to denote any of this spaces.
  \end{enumerate}
\end{prop}

%
%

\section{Malliavin calculus and time reversal}
\label{sec_stoch-integr}

Our reference probability space is $\Omega=\CC([0,T],\, \R^n),$ the
space of $\R^n$-valued, continuous functions, null at time $0$. The
Cameron-Martin space is denoted by $\H$ and is defined as \begin{math}
  \H=I^1_{0^+}(\L^2([0,T])).
\end{math}
In what follows, the space $\L^2([0,\, T])$ is identified with its
topological dual. We denote by $\kappa$ the canonical embedding from
$\H$ into $\Omega$.  The probability measure $\P$ on $\Omega$ is such
that the canonical map $W\, :\, \omega \mapsto (\omega(t), \, t\in
[0,\, T])$ defines a standard $n$-dimensional Brownian motion. A
mapping $\phi$ from $\Omega$ into some separable Hilbert space
${\mathfrak H}$ is called cylindrical if it is of the form
$\phi(w)=\sum_{i=1}^df_i(\<v_{i,1},w\>,\cdots,\<v_{i,n},w\>) x_i$
where for each $i,$ $f_i\in {\mathcal C}_0^\infty (\RE^n, \RE)$ and
$(v_{i,j},\, j=1,\, \cdots,\, n)$ is a sequence of $\Omega^*.$ For
such a function we define $\n^\WW \phi$ as
$$
\n^\WW \phi(w)=\sum_{i,j=1} \partial_j
f_i(\<v_{i,1},w\>,\cdots,\<v_{i,n},w\>){\tilde{v}}_{i,j}\otimes x_i,
$$
where $\tilde{v}$ is the image of $v\in \Omega^*$ by the map
$(I^1_{0^+}\circ \kappa)^*.$ From the quasi-invariance of the Wiener
measure \cite{ustunel_book}, it follows that $\n^\WW $ is a closable
operator on $L^p(\Omega;{\mathfrak H})$, $p\geq 1$, and we will denote
its closure with the same notation. The powers of $\n^\WW $ are
defined by iterating this procedure. For $p>1$, $k\in \N$, we denote
by $\D_{p,k}({\mathfrak H})$ the completion of ${\mathfrak H}$-valued
cylindrical functions under the following norm
$$
\|\phi\|_{p,k}=\sum_{i=0}^k \|(\n^\WW) ^i\phi\|_{L^p(\Omega;\
  {\mathfrak H}\otimes \L^p([0,1])^{\otimes i})}\,.
$$
We denote by $\LL_{p,1}$ the space $\D_{p,1}(\L^p([0,\, T];\, \R^n)).$
The divergence, denoted $\div^\WW$ is the adjoint of $\n^\WW $: $v$
belongs to $\Dom_p \div^\WW$ whenever for any cylindrical $\phi,$
\begin{equation*}
  |\esp{\int_0^T v_s \n^\WW _s\phi\d s}|\le \, c \lVert \phi\rVert_{L^p}
\end{equation*}
and for such a process $v,$ $$\esp{\int_0^T v_s \n^\WW _s\phi\d
  s}=\esp{ \phi\, \div^\WW v}.$$
%
%
We introduced the temporary notation $W$ for standard Brownian motion to
clarify the forthcoming distinction between a standard Brownian motion
and its time reversal. Actually, the time reversal of a standard
Brownian is also a standard Brownian motion and thus, both of them
``live'' in the same Wiener space. We now precise how their respective
Malliavin gradient and divergence are linked. Consider $B=(B(t),\,
t\in [0,\, T])$ an $n$-dimensional standard Brownian motion and
$\check{B}^T=(B(T)-B(T-t),\, t\in [0,\, T])$ its time reversal.
Consider the following map \begin{align*}
  \Theta_T \, : \, \Omega & \longrightarrow \Omega\\
  \omega & \longmapsto \comega=\omega(T)-\taut\omega ,\end{align*} and
the commutative diagram \begin{equation*}
  \begin{CD}
    \L^2 @>\taut >> \L^2\\
    @V{I^1_{0^+}}VV @VV{I^1_{0^+}}V\\
    \Omega \supset \H\qquad @>>\Theta_T>\qquad \H\subset \Omega
  \end{CD}
\end{equation*}
Note that $\Theta_T^{-1}=\Theta_T$ since $\omega(0)=0.$ For a function
$f\in {\mathcal C}^\infty_b(\R^{nk})$, we define
\begin{align*}
  \n_r f(\omega(t_1),\cdots,\, \omega(t_k))&=\sum_{j=1}^k\partial_j f
  (\omega(t_1),\cdots,\, \omega(t_k)) \car_{[0,\, t_j]}(r)\text{ and }\\
  \cn_r f(\comega(t_1),\cdots,\, \comega(t_k))&=\sum_{j=1}^k\partial_j
  f (\comega(t_1),\cdots,\, \comega(t_k)) \car_{[0,\, t_j]}(r).
\end{align*} The operator $\n=\n^B$ (respectively
$\cn=\n^{\check{B}}$) is the Malliavin gradient associated with a
standard Brownian motion (respectively its time reversal). Since,
\begin{equation*}
  f(\comega(t_1),\cdots,\, \comega(t_k))=f(\omega(T)-\omega(T-t_1),\cdots,\, \omega(T)-\omega(T-t_k)),
\end{equation*}
we can consider $f(\comega(t_1),\cdots, \,\comega(t_k))$ as a
cylindrical function with respect to the standard Brownian motion. As
such its gradient is given by
\begin{equation*}
  \n_r f(\comega(t_1),\cdots,\, \comega(t_k))=\sum_{j=1}^k\partial_j f
  (\comega(t_1),\cdots, \comega(t_k))\car_{[T-t_j,\, T]}(r).
\end{equation*}
We thus have, for any cylindrical function $F$,
\begin{equation}
  \label{eq:4}
  \n F\circ \Theta_T(\omega)=\taut \cn F(\comega).
\end{equation}
Since $\Theta^*_T\P=\P$ and $\taut$ is continuous from $\L^p$ into
itself for any $p$, it is then easily shown that the spaces $\D_{p,\,
  k}$ and $\check{\D}_{p, \, k}$ (with obvious notations) coincide for
any $p,\, k$ and that \eqref{eq:4} holds for any element of one of
these spaces. Hence we have proved the following theorem:
\begin{theorem}
  \label{thm:cdelta}
  For any $p\ge 1$ and any integer $k$, the spaces $\D_{p,\, k}$ and
  $\check{\D}_{p, \, k}$ coincide. For any $F\in \D_{p,\, k}$ for some
  $p,\, k$,
  \begin{equation*}
    \n (F\circ \Theta_T)=\taut \cn (F\circ \Theta_T), \P \text{ a.s..}
  \end{equation*}
\end{theorem}
By duality, an analog result follows for divergences.
\begin{theorem}
  \label{thm:taut_delta}
  A process $u$ belongs to the domain of $\delta$ if and only if
  $\taut u$ belongs to the domain of $\check{\delta}$ and then, the
  following equality holds:
  \begin{equation}
    \label{eq:7}
    \check{\delta}(u(\comega))(\comega)=\delta(\taut
    u(\comega))(\omega)=\delta(\taut u\circ \Theta_T)(\omega).
  \end{equation}
\end{theorem}
\begin{proof}
  For $u\in \L^2$, for cylindrical $F$, we have on the one hand:
  \begin{equation*}
    \esp{F(\comega)\cdelta u(\comega)}=\esp{( \cn F(\comega),\, u)_{\L^2}},
  \end{equation*}
  and on the other hand,
  \begin{align*}
    \esp{( \cn F(\comega),\, u)_{\L^2}}&=\esp{(\taut \n
      F\circ \Theta_T(\omega),\, u)_{\L^2}}\\
    &=\esp{(\n F\circ \Theta_T(\omega),\, \taut u)_{\L^2}}\\
    &=\esp{F\circ \Theta_T(\omega)\delta(\taut u)(\omega)}\\
    &=\esp{F(\comega)\delta(\taut u)(\omega)}.
  \end{align*}
  Since this is valid for any cylindrical $F$, (\ref{eq:7}) holds for
  $u\in \L^2$. Now, for $u$ in the domain of divergence (see
  \cite{nualart.book,ustunel_book}),
  \begin{equation*}
    \delta u = \sum_{i} \Bigl( (u,\, h_i)_{\L^2}\delta h_i-(\n
    u,h_i\otimes h_i)_{\L^2\otimes\L^2}\Bigr),
  \end{equation*}
  where $(h_i,\, i\in \N)$ is an orthonormal basis of $\L^2([0,\,
  T];\, \R^n)$. Thus, we have
  \begin{align*}
    \cdelta (u(\comega))(\comega) &= \sum_{i} \Bigl( (u(\comega),\,
    h_i)_{\L^2}\cdelta h_i(\comega)-(\cn
    u(\comega),h_i\otimes h_i)_{\L^2\otimes\L^2}\Bigr)\\
    &=\sum_{i} \Bigl( (u(\comega),\, h_i)_{\L^2}\delta(\taut
    h_i)(\omega)-(\n
    u(\comega),\taut h_i\otimes h_i)_{\L^2\otimes\L^2}\Bigr)\\
    &=\sum_{i} \Bigl( (\taut u(\comega),\, \taut
    h_i)_{\L^2}\delta(\taut h_i)(\omega)-(\n \taut u(\comega),\taut
    h_i\otimes \taut h_i)_{\L^2\otimes\L^2}\Bigr),
  \end{align*}
  where we have taken into account that $\taut$ in an
  involution. Since $(h_i,\, i\in \N)$ is an orthonormal basis of
  $\L^2([0,\, T];\, \R^n)$, identity (\ref{eq:7}) is satisfied for any
  $u$ in the domain of $\delta$.
\end{proof}

\subsection{Causality and quasi-nilpotence}
\label{sec:caus-quasi-nilp}
In anticipative calculus, the notion of trace of an operator plays a
crucial role, we refer to \cite{0635.47002} for more details on trace.
\begin{defn}
  Let $V$ be a bounded map from $\L^2([0,\, T];\, \R^n)$ into
  itself. The map $V$ is said to be trace-class whenever for one CONB
  $(h_n,\, n\ge 1)$ of $\L^2([0,\, T];\, \R^n)$,
  \begin{equation*}
    \sum_{n\ge 1} |(Vh_n,\, h_n)_{\L^2}|\text{ is finite.}
  \end{equation*}
  Then, the trace of $V$ is defined by
  \begin{equation*}
    \trace(V)= \sum_{n\ge 1} (Vh_n,\, h_n)_{\L^2}.
  \end{equation*}
\end{defn}
It is easily shown that the notion of trace does not depend on the
choice of the CONB.
\begin{defn}
  A family $E$ of projections $(E_\lambda$, $\lambda \in [0,1])$ in  $\L^2([0,\, T];\, \R^n)$ is called a resolution of the identity if it
  satisfies the conditions
  \begin{enumerate}
  \item $E_0=0$ and $E_1=\id$.
  \item $E_\lambda E_\mu=E_{\lambda\wedge \mu}$.
  \item $\lim_{\mu\downarrow \lambda}E_\mu =E_\lambda$ for any
    $\lambda\in [0,\, 1)$ and $\lim_{\mu\uparrow 1}E_\mu=\id.$
  \end{enumerate}
\end{defn}
For instance, the family $E=(e_{\lambda T}, \, \lambda\in [0,\, 1])$
is a resolution of the identity in $\L^2([0,\, T];\, \R^n).$
\begin{defn}
  A partition $\pi$ of $[0,\, T]$ is a sequence $\{0=t_0<t_1<\ldots
  <t_n=T\}$. Its mesh is denoted by $|\pi|$ and defined by
  $|\pi|=\sup_{i}|t_{i+1}-t_i|.$ 
\end{defn}
The causality plays a crucial role in what follows. The next
definition is just the formalization in terms of operator of the
intuitive notion of causality.
\begin{defn}
  A continuous map $V$ from $\L^2([0,\, T];\, \R^n)$  into itself is said
  to be $E$-causal if and only if the following condition holds:
  \begin{equation*}
    E_\lambda VE_\lambda = E_\lambda V \text{ for any } \lambda \in
    [0,\, 1].
  \end{equation*}
\end{defn}
For instance, an operator $V$ in
integral form $ Vf(t)=\int_0^T V(t,s)f(s)\d s$ is causal if and only
if $V(t,s)=0$ for $s\ge t$, i.e., computing $Vf(t)$ needs only the
knowledge of $f$ up to time $t$ and not after. Unfortunately, this
notion of causality is insufficient for our purpose and we are led to
introduce the notion of strict causality as in \cite{MR663906}.
\begin{defn}
  Let $V$ be a causal operator. It is a strictly causal operator
  whenever for any $\ee>0$, there exists a partition $\pi$ of $[0,T]$
  such that for any $\pi^\prime=\{0=t_0<t_1<\ldots<t_n=T\}\subset \pi$,
  \begin{equation*}
    \| (E_{t_{i+1}}-E_{t_i})V (E_{t_{i+1}}-E_{t_i})\|_{\L^2} < \ee,
    \text{ for } i=0,\cdots,\, n-1.
  \end{equation*}
\end{defn}
Note carefully that the identity map is causal but not strictly
causal. Indeed, if $V=\Id$, for any $s<t$,
\begin{equation*}
  \| (E_{t}-E_{s})V (E_{t}-E_{s})\|_{\L^2}= \|E_{t}-E_{s}\|_{\L^2}=1
\end{equation*}
since $E_{t}-E_{s}$ is a projection. However, if $V$
is hyper-contractive, we have the following result:
\begin{lemma}
  \label{lem:strict-causality}
   Assume the resolution of the
  identity to be either $E=(e_{\lambda T}, \, \lambda\in[0,\,1])$ or
  $E=(\Id-e_{(1-\lambda) T}, \, \lambda\in[0,\,1])$.  If $V$ is an
  $E$-causal map continuous from $\L^2$ into $\L^p$ for some $p>2$
  then $V$ is strictly $E$-causal.
\end{lemma}
\begin{proof}Let $\pi$ be any partition of $[0,\, T].$ Assume
  $E=(e_{\lambda T}, \, \lambda\in[0,\,1])$, the very same proof
  works for the other mentioned resolution of the identity.  According
  to Hölder formula, we have: For any $ 0\le s<t\le T$,
  \begin{align*}
    \| (E_t-E_s) V (E_t-E_{s}) f\|_{\L^2}&=\int_s^{t}  |V(f \car_{(s,\, t]})(u)|^2\d u\\
    &\le (t-s)^{1-2/p} \|  V(f \car_{(s,\, t]})\|_{\L^{p/2}}\\
    &\le c\, (t-s)^{1-2/p} \|f\|_{\L^2}.
  \end{align*}
  Then, for any $\ee>0$, there exists $\eta>0$ such that $|\pi|<\eta $
  implies $\|  (E_{t_{i+1}}-E_{t_i})V (E_{t_{i+1}}-E_{t_i})
  f\|_{\L^2}\le \ee$ for any  $\{0=t_0<t_1<\ldots<t_n=T\}\subset \pi$
  and any $i=0,\cdots,\, n-1.$
\end{proof}
The importance of strict causality lies in the next theorem we borrow
from \cite{MR663906}.
\begin{theorem}
  \label{thm:nilpotent}
  The set of strictly causal operators coincides with the set of
  quasi-nilpotent operators, i.e., trace-class operators such that
  $\trace(V^n)=0$ for any integer $n\ge 1$.
\end{theorem}
Moreover, we have the following stability theorem.
\begin{theorem}
  \label{thm:stability}
  The set of strictly causal operators is a two-sided ideal in the set
  of causal operators.
\end{theorem}
\begin{defn} Let $E$ be a resolution of the identity in  $\L^2([0,\, T];\, \R^n)$. Consider the filtration $\F^E$
  defined as
  \begin{equation*}
    \F^E_t=\sigma\{\div^\WW (E_\lambda h), \lambda\le t,\, h\in \L^2\}.
  \end{equation*}
  An $\L^2$-valued random variable $u$ is said to be $\F^E$-adapted if
  for any $h\in \L^2$, the real valued process $<E_\lambda u,\, h>$ is
  $\F^E$-adapted.  We denote by $\D_{p,k}^E({\mathfrak H})$ the set of
  $\F^E$-adapted random variables belonging to $\D_{p,k}( {\mathfrak
    H}).$
\end{defn}
If $E=(e_{\lambda T}, \, \lambda \in [0,1])$, the notion of $\F^E$
adapted processes coincides with the usual one for the Brownian filtration
and it is well known that a process $u$ is adapted if and only if
$\n^\WW _r u(s)=0$ for $r>s.$ This result can be generalized to any
resolution of the identity.
\begin{theorem}[Proposition 3.1 of \protect\cite{ustunel95_1}]
  Let $u$ belongs to $\LL_{p,1}$. Then $u$ is $\F^E$-adapted if and
  only if $\n^\WW u$ is $E$-causal.
\end{theorem}
We then have the following key theorem:
\begin{theorem}
  \label{thm:trace_zero}Assume the resolution of the identity to be
  $E=(e_{\lambda T}, \, \lambda\in[0,\,1])$ either
  $E=(\Id-e_{(1-\lambda) T}, \, \lambda\in[0,\,1])$ and that $V$
  is an $E$-strictly causal continuous operator from $\L^2$ into $\L^p$ for some $p>2$. Let $u$ be an element of $\D_{2,1}^E(\L^2).$
  Then, $V\n^\WW u $ is of trace class and we have $\trace(V\n^\WW
  u)=0$.
\end{theorem} \begin{proof} Since $u$ is adapted, $\n^\WW u$ is
  $E$-causal. According to Theorem \ref{thm:stability}, $V\n^\WW u$ is
  strictly causal and the result follows by Theorem
  \ref{thm:nilpotent}.
\end{proof}
In what follows, $\E$ is the resolution of the identity in the Hilbert
space $\L^2$ defined by $e_{\lambda\, T} f=f\car_{[0,\, \lambda T]}$
and $\cE$ is the resolution of the identity defined by \begin{math}
  \check{e}_{\lambda T}f = f\car_{[(1-\lambda)T,T]}.
\end{math}
The filtration $\F^\E$ and $\F^\cE$ are defined accordingly.  Next
lemma is immediate when $V$ is given in the form $Vf(t)=\int_0^t
V(t,\, s) f(s)\d s$. Unfortunately such a representation as an
integral operator is not always available. We give here an algebraic
proof to emphasize the importance of causality.
\begin{lemma} \label{lem:causalite} Let $V$ be a map from $\L^2([0,\,
  T];\, \R^n)$ into itself such that $V$ is $\E$-causal. Let $V^*$ be
  the adjoint of $V$ in $\L^2([0,\, T];\, \R^n)$. Then, the map $\taut
  V^*_T\taut$ is $\cE$-causal.
\end{lemma}
\begin{proof}
  This is a purely algebraic lemma once we have noticed that
  \begin{equation}
    \label{eq:2}
    \taut
    e_r=(\id-e_{T-r})\taut \text{ for any } 0\le r \le T.
  \end{equation}
  For, it suffices to write
  \begin{multline}\label{eq:8}
    \taut e_r f(s)=f(T-s)\car_{[0,\, r]}(T-s)\\
    =f(T-s)\car_{[T-r,\, T]}(s)=(\id-e_{T-r})\taut f(s), \text{ for
      any } 0\le s \le T.
  \end{multline}
  We have to show that
  \begin{equation*}
    e_r \taut V^*_T \taut e_r=e_r \taut V^*_T \taut  \text{ or
      equivalently } e_r \taut V \taut e_r =\taut V \taut e_r,
  \end{equation*}
  since $e_r^*=e_r$ and $\taut^*=\taut.$ Now, \eqref{eq:8} yields
  \begin{equation*}
    e_r \taut V \taut e_r = \taut V \taut e_r- e_{T-r} V \taut e_r.
  \end{equation*}
  Use \eqref{eq:8} again to obtain
  \begin{equation*}
    e_{T-r} V \taut e_r=e_{T-r} V (\Id -e_{T-r})\taut=(e_{T-r} V -e_{T-r}  V e_{T-r})\taut=0,
  \end{equation*}
  since $V$ is $E$-causal.
\end{proof}
%
%

\subsection{Stratonovitch integrals}
\label{sec:stoch-integr-with}
In what follows, $\eta$ belongs to $(0,1]$ and $V$ is a linear
operator. For any $p\ge 2$, we set: \begin{hyp}[$p,\, \eta$]
  \label{hypA}
  The linear map $V$ is continuous from $\L^p([0,\, T]; \R^n)$ into
  the Banach space $\Hol(\eta)$.
\end{hyp}
\begin{defn} Assume that Hypothesis \ref{hypA}($p,\, \eta$) holds. The
  Volterra process associated to $V$, denoted by $W^V$ is defined by
  \begin{equation*}
    W^V(t)=\div^\WW \bigl(V(\car_{[0,\, t]})\bigr), \text{ for all } t \in [0, \, T].
  \end{equation*}
\end{defn}
For any subdivision $\pi$ of $[0,\, T]$, i.e., $\pi =\{0=t_0<t_1<
\ldots<t_n=T\}$, of mesh $|\pi|$,
we consider the Stratonovitch sums: \begin{multline}
  R^\pi(t,u)=\div^\WW\Bigl(\sum_{t_i\in\pi}\frac{1}{\theta_i}
  \int_{t_i\wedge t}^{t_{i+1}\wedge t} \kern-3pt Vu(r) \d r\, \car_{[t_i,\, t_{i+1})}\Bigr)\\
  + \sum_{t_i\in\pi}\frac{1}{\theta_i} \iint\limits_{[t_i\wedge
    t,t_{i+1}\wedge t]^2} V(\n^\WW _ru)(s) \d s\d r.
  \label{eq:def_de_RpiT}
\end{multline}
\begin{defn}
  \label{def:strat-integr}
  We say that $u$ is $V$-Stratonovitch integrable on $[0,t]$ whenever
  the family $\text{R}^\pi(t,u),$ defined in (\ref{eq:def_de_RpiT}),
  converges in probability as $|\pi|$ goes to $0.$ In this case the
  limit will be denoted by $\int_0^t u(s) \circ \d W^V(s).$ \end{defn}
\begin{example}
  \label{fbm_levy}
  The first example is the so-called L\'evy fractional Brownian motion
  of Hurst index $H>1/2$, defined as
  \begin{displaymath}
    \frac{1}{\Gamma(H+1/2)} \int_0^t (t-s)^{H-1/2}\d B_s=\div(I_{T^-}^{H-1/2}(\car_{[0,\, t]})).
  \end{displaymath}
  This amounts to say that $V=I_{T^-}^{H-1/2}.$ Thus Hypothesis
  \ref{hypA}$(p,H-1/2-1/p)$ holds provided $p(H-1/2)>1$.
\end{example}
\begin{example}
  \label{fbm}
  The other classical example is the fractional Brownian motion with
  stationary increments of Hurst index $H>1/2,$ which can be written
  as
  \begin{equation*}
    \int_0^t K_H(t,s)\d B(s),
  \end{equation*}
  where
  \begin{equation}
\label{defdekh}
    K_H(t,r)=\frac{(t-r)^{H- \frac{1}{2}}}{\Gamma(H+\frac{1}{2})}
    F(\frac{1}{2}-H,H-\frac{1}{2}, H+\frac{1}{2},1- 
    \frac{t}{r})1_{[0,t)}(r).
  \end{equation}
  The Gauss hyper-geometric function $F(\alpha,\beta,\gamma,z)$ (see
  \cite{nikiforov88}) is the analytic continuation on ${\mathbb
    C}\times {\mathbb C}\times {\mathbb C} \backslash \{-1,-2,\ldots
  \}\times \{z\in {\mathbb C}, Arg |1-z| < \pi\}$ of the power series
  \begin{displaymath}
    \sum_{k=0}^{+ \infty} \frac{(\alpha)_k(\beta)_k}{(\gamma)_k 
      k!}z^k,
  \end{displaymath}
  and
  \begin{displaymath}
    (a)_0=1 \text{ and } (a)_k = 
    \frac{\Gamma(a+k)}{\Gamma(a)}=a(a+
    1)\dots (a+k-1).
  \end{displaymath}
  We know from \cite{samko93} that $K_H$ is an isomorphism from $\L^p
  ([0,1])$ onto $\I_{H+1/2,p}^+$ and
  \begin{equation*}
    K_Hf = 
    I_{0^+}^1x^{H-\half}I_{0^+}^{H-\half}x^{\half-H}f .
  \end{equation*}
  Consider $\K_H=I_{0^+}^{-1}\circ K_H.$ Then it is clear that
  \begin{equation*}
    \int_0^t K_H(t,\, s)\d B(s)=\int_0^t (\K_H)_T^*(\car_{[0,t]})(s)\d B(s),
  \end{equation*}
  hence that we are in the framework of Definition
  \ref{def:strat-integr} provided that we take
  $V=(\K_H)_T^*$. Hypothesis \ref{hypA}$(p,H-1/2-1/p)$ is satisfied
  provided that $p(H-1/2)>1.$
\end{example}

The next theorem then follows from \cite{decreusefond03.1}.
\begin{theorem}
  \label{thm:strato_integrable}
  Assume that Hypothesis \ref{hypA}($p,\, \eta$) holds. Assume that
  $u$ belongs to $\LL_{p,1}.$ Then $u$ is $V$-Stratonovitch
  integrable, there exists a process which we denote by $D^\WW u$ such
  that $D^\WW u$ belongs to $L^p(\P\otimes d s)$ and
  \begin{equation}
    \label{eq:6}
    \int_0^T u(s) \circ \d W^V(s)=\div^\WW (Vu)+ \int_0^T D^\WW u(s)\d s.
  \end{equation}
  The so-called ``trace-term'' satisfies the following estimate:
  \begin{equation}
    \esp{  \int_0^T |D^\WW u(r)|^p \d r}\le \, c\, T^{p\eta}\| u\|_{\LL_{p,\, 1}}^p,\label{eq:121}
  \end{equation}
  for some universal constant $c$.  Moreover, for any $r\le T$, $e_ru$
  is $V$-Stratonovitch integrable and
  \begin{equation*}
    \int_0^r u(s) \circ \d W^V(s)= \int_0^T (e_ru)(s) \circ \d W^V(s)=\div^\WW (Ve_ru)+ \int_0^r D^\WW u(s)\d s
  \end{equation*}
  and we have the maximal inequality:
  \begin{equation}
    \label{eq:3}
    \esp{    \|\int_0^. u(s)\circ \d W^V(s)\|_{\Hol(\eta)}^p}\le \, c\, \|u\|_{\LL_{p,1}}^p,
  \end{equation}
  where $c$ does not depend on $u$.
\end{theorem}
The main result of this Section is the following theorem which states
that the time reversal of a Stratonovitch integral is an adapted
integral with respect to the time reversed Brownian motion. Due to
its length, its proof is postponed to Section \ref{sec:substitution-formula}.
\begin{theorem}
  \label{thm:int_strato_inverse}
  Assume that Hypothesis \ref{hypA}($p,\, \eta$) holds.  Let $u$
  belong to $\LL_{p,1}$ and let $\VV=\tau_t V\tau_T$.  Assume
  furthermore that $V$ is $\check{E}_0$-causal and that $\cu=u\circ
  \Theta_T^{-1}$ is $\F^{\check{E}_0}$-adapted. Then,
  \begin{equation}
    \label{eq:9}
    \int_{T-t}^{T-r} \taut u(s)\circ \d W^V(s) =\int_r^t \VV(\car_{[r,\, t]} \check{u})(s)\d \check{B}^T(s), \ 0\le r\le t\le T,
  \end{equation}
  where the last integral is an Itô integral with respect to the time
  reversed Brownian motion
  $\check{B}^T(s)=B(T)-B(T-s)=\Theta_T(B)(s).$
\end{theorem}

\begin{rem}
  Note that at a formal level, we could have an easy proof of this
  theorem. For instance, consider the L\'evy fBm, a simple
  computations shows that $\VV=I^{H-1/2}_{0^+}$ for any $T.$ Thus, we
  are led to compute $\trace(I^{H-1/2}_{0^+}\n u)$. If we had
  sufficient regularity, we could write
  \begin{equation*}
    \trace(I^{H-1/2}_{0^+}\n u)=\int_0^T \int_0^s (s-r)^{H-3/2}\n_s u(r)\d r\d s=0 ,
  \end{equation*}
  since $\n_s u(r)=0$ for $s>r$ for $u$ adapted. Obviously, there are
  many flaws in these lines of proof: The operator $I^{H-1/2}_{0^+}\n
  u$ is not regular enough for such an expression of the trace to be
  true. Even more, there is absolutely no reason for $\VV\n u$ to be a
  kernel operator so we can't hope such a formula. These are the
  reasons that we need to work with operators and not with kernels.
\end{rem}
\section{Volterra driven SDEs}
\label{sec_time-reversal}
Let $\G$ the group of homeomorphisms of $\R^n$ equipped with the
distance: We introduce a distance $d$ on $\G$ by
\begin{equation*}
  d(\varphi,\, \phi)=\rho(\varphi,\, \phi)+\rho(\varphi^{-1},\, \phi^{-1}),
\end{equation*}
where
\begin{equation*}
  \rho(\varphi,\, \phi)=\sum_{N=1}^\infty 2^{-N} \frac{\sup_{|x|\le N}|\varphi(x)-\phi(x)|}{1+\sup_{|x|\le N}|\varphi(x)-\phi(x)|}\cdotp
\end{equation*}
Then, $\G$ is a complete topological group.  Consider the equations
\begin{equation}
  \label{eq:Ap}
  \tag{A}
  X_{r,\, t}=x+\int_r^t \sigma(X_{r,\, s})\circ \d W^V(s), \ 0\le r\le
  t\le T.
\end{equation}
\begin{equation}
  \label{eq:B1}
  \tag{B}
  Y_{r,\, t} =x-\int_r^t \sigma(Y_{s,\, t})\circ \d W^V(s),  \ 0\le r\le
  t\le T.
\end{equation}
As a solution of \eqref{eq:Ap} is to be constructed by ``inverting'' a solution
of \eqref{eq:B1}, we need to add to the definition of a solution of \eqref{eq:Ap} or
\eqref{eq:B1} the requirement of being a flow of homeomorphisms. This is
the meaning of the following definition.

\begin{defn}
  By a solution of \eqref{eq:A}, we mean a measurable map
  \begin{align*}
    \Omega \times [0,T] \times [0,T] & \longrightarrow \G\\
    (\omega,\, r,\, t) & \longmapsto (x\mapsto X_{r,t}(\omega,x))
  \end{align*}
  such that the following properties are satisfied~:
  \begin{enumerate}
  \item For any $0\le r \le t \le T$, for any $x\in \R^n$, $X_{r,\,
      t}(\omega,\, x)$ is $\sigma\{W^V(s), r\le s\le t\}$-measurable,
  \item For any $0\le r \le T$, for any $x\in \R^n$, the processes
    $(\omega,\, t)\mapsto X_{r,t}(\omega,\, x)$ and $(\omega,\,
    t)\mapsto X_{r,t}^{-1}(\omega,\, x)$ belong to $\LL_{p,1}$ for
    some $p\ge 2$.
  \item For any $0\le r\le s\le t,$ for any $x\in \R^n$, the following
    identity is satisfied:
    \begin{equation*}
      X_{r,t}(\omega,\, x)=X_{s,t}(\omega,\, X_{r,s}(\omega,\, x)).
    \end{equation*}
  \item Equation \eqref{eq:A} is satisfied for any $0\le r\le t\le T$
    $\P$-a.s.. \end{enumerate}
\end{defn}
\begin{defn}
  By a solution of \eqref{eq:B}, we mean a measurable map
  \begin{align*}
    \Omega \times [0,T] \times [0,T] & \longrightarrow \G\\
    (\omega,\, r,\, t) & \longmapsto (x\mapsto Y_{r,t}(\omega,x))
  \end{align*}
  such that the following properties are satisfied~:
  \begin{enumerate}
  \item For any $0\le r \le t \le T$, for any $x\in \R^n$, $Y_{r,\,
      t}(\omega,\, x)$ is $\sigma\{W^V(s), r\le s\le t\}$-measurable,
  \item For any $0\le r \le T$, for any $x\in \R^n$, the processes
    $(\omega,\, r)\mapsto Y_{r,t}(\omega,\, x)$ and $(\omega,\,
    r)\mapsto Y_{r,t}^{-1}(\omega,\, x)$ belong to $\LL_{p,1}$ for
    some $p\ge 2$.
  \item Equation \eqref{eq:B} is satisfied for any $0\le r\le t\le T$
    $\P$-a.s..
  \item For any $0\le r\le s\le t,$ for any $x\in \R^n$, the following
    identity is satisfied:
    \begin{equation*}
      Y_{r,t}(\omega,\, x)=Y_{r,s}(\omega,\, Y_{s,t}(\omega,\, x)).
    \end{equation*} 
  \end{enumerate}
\end{defn}
At last consider the equation, for any $0\le r\le t\le T$,
\begin{equation}
  \label{eq:C}
  \tag{C}
  Z_{r,\, t}=x-\int_r^t \VV( \sigma\circ Z_{.,t}\ \car_{[r,t]})(s)\d \check{B}^T(s)
\end{equation}
where $B$ is a standard $n$-dimensional Brownian motion.
\begin{defn}
  By a solution of \eqref{eq:C}, we mean a measurable map
  \begin{align*}
    \Omega \times [0,T] \times [0,T] & \longrightarrow \G\\
    (\omega,\, r,\, t) & \longmapsto (x\mapsto Z_{r,t}(\omega,x))
  \end{align*}
  such that the following properties are satisfied~:
  \begin{enumerate}
  \item For any $0\le r \le t\le T$, for any $x\in \R^n$, $Z_{r,\,
      t}(\omega,\, x)$ is $\sigma\{\check{B}^T (s), s\le r\le t
    \}$-measurable,
  \item For any $0\le r \le t \le T$, for any $x\in \R^n$, the
    processes $(\omega,\, r)\mapsto Z_{r,t}(\omega,\, x)$ and
    $(\omega,\, r)\mapsto Z_{r,t}^{-1}(\omega,\, x)$ belong to
    $\LL_{p,1}$ for some $p\ge 2$.
  \item Equation \eqref{eq:C} is satisfied for any $0\le r\le t\le T$
    $\P$-a.s..
  \end{enumerate}
\end{defn}

\begin{theorem}
  \label{thm:solution_C}
  Assume that $\VV$ is an $E^0$ causal map continuous from $\L^p$ into
  $\I_{\alpha,p}$ for $\alpha >0$ and $p\ge 4$ such that $\alpha p>1.$
  Assume $\sigma$ is Lipschitz continuous and sub-linear, see Eqn.
  \eqref{eq:22} for the definition.  Then, there exists a unique
  solution to equation (\ref{eq:C}).  Let $Z$ denote this
  solution. For any $(r,\, r^\prime)$,
  \begin{equation*}
    \esp{|Z_{r,T}-Z_{r^\prime, T}|^p}\le c |r-r^\prime|^{p\eta}.
  \end{equation*}
  Moreover,
  \begin{equation*}
    (\omega,\, r)\mapsto Z_{r,s}(\omega, Z_{s,t}(\omega,\, x))\in
    \LL_{p,1}, \text{ for any } r\le s\le t\le T.
  \end{equation*}
\end{theorem}
Since this proof needs several lemmas, we defer it to Section
\ref{sec:equation-refeq:c}.


\begin{theorem}
  \label{thm:correspondance_solution_B_C}
  Assume that $\VV$ is an $E^0$ causal map continuous from $\L^p$ into
  $\I_{\alpha,p}$ for $\alpha >0$ and $p\ge 2$ such that $\alpha p>1.$
  For fixed $T$, there exists a bijection between the space of
  solutions of Equation (\ref{eq:B1}) on $[0,T]$ and the set of
  solutions of Equation (\ref{eq:C}).
\end{theorem}
\begin{proof}
  Set
  \begin{equation*}
    Z_{r,T}(\comega,\, x)=Y_{T-r,T}(\Theta_T^{-1}(\comega),\, x)
  \end{equation*}
  or equivalently
  \begin{equation}\label{eq:20}
    Y_{r,T}(\omega,\, x)=Z_{T-r,T}(\Theta_T(\omega),\, x).
  \end{equation}
  According to Theorem \ref{thm:int_strato_inverse}, $Y$ is satisfies
  (\ref{eq:B1}) if and only if $Z$ satisfies (\ref{eq:C}). The
  regularity properties are immediate since $\L^p$ is stable by
  $\taut.$
\end{proof}
The first part of the next result is then immediate.
\begin{corollary}
  \label{thm:solution_B} Assume that $\VV$ is an $E^0$ causal map
  continuous from $\L^p$ into $\I_{\alpha,p}$ for $\alpha >0$ and $p\ge
  2$ such that $\alpha p>1.$ Then Equation (\ref{eq:B1}) has one and
  only solution and for any $0\le r\le s\le t,$ for any $x\in \R^n$,
  the following identity is satisfied:
  \begin{equation}\label{eq:14}
    Y_{r,t}(\omega,\, x)=Y_{r,s}(\omega,\, Y_{s,t}(\omega,\, x)).
  \end{equation}
\end{corollary}
\begin{proof}
  According to Theorem \ref{thm:correspondance_solution_B_C} and
  \ref{thm:solution_C}, \eqref{eq:B1} has at most one solution since
  \eqref{eq:C} has a unique solution. As to the existence, point (1)
  to (3) are immediately deduced from the corresponding properties of
  $Z$ and Equation \eqref{eq:20}.

  According to Theorem \ref{thm:solution_C}, $(\omega,\, r)\mapsto
  Y_{r,s}(\omega,\, Y_{s,t}(\omega,\, x))$ belongs to $\LL_{p,1}$
  hence we can apply the substitution formula and we get:
  \begin{multline}\label{eq:13}
    Y_{r,s}(\omega,\, Y_{s,t}(\omega,\, x))=Y_{s,t}(\omega,\, x)-\left.\int_r^s \sigma(Y_{\tau,s}(\omega,\, x))\circ \d W^V(\tau)\right|_{ x=Y_{s,t}(\omega,\, x)}\\
    =x-\int_s^t \sigma(Y_{\tau,t}(\omega,\, x)\circ \d W^V(\tau)\\
    -\int_r^s \sigma(Y_{\tau,s}(\omega,\, Y_{s,t}(\omega,\, x)))\circ
    \d W^V(\tau).
  \end{multline}
  Set
  \begin{equation*}
    R_{\tau,t}=
    \begin{cases}
      Y_{\tau,t}(\omega,\, x)& \text{ for } s\le \tau \le t\\
      Y_{\tau,s}(\omega,\, Y_{s,t}(\omega,\, x)) & \text{ for } r\le
      \tau \le s.
    \end{cases}
  \end{equation*}
  Then, in view of (\ref{eq:13}), $R$ appears to be the unique
  solution (\ref{eq:B1}) and thus $R_{s,t}(\omega,\,
  x)=Y_{s,t}(\omega,\, x).$ Point (4) is thus proved.
\end{proof}
\begin{corollary}
  For $x$ fixed, the random field $(Y_{r,t}(x),\, 0\le r\le t\le T) $
  admits a continuous version. Moreover,
  \begin{equation*}
    \esp{|Y_{r,s}(x)-Y_{r^\prime,s^\prime}(x)|^p}\le c (1+|x|^p) ( |s^\prime-s|^{p\eta}+|r-r^\prime|^{p\eta}).
  \end{equation*}
  We still denote by $Y$ this continuous version.
\end{corollary}
\begin{proof}
  Without loss of generality, assume that $s\le s^\prime$ and remark that $Y_{s,\,
    s^\prime(x)}$ thus belongs to $\sigma\{\check{B}^T_u, \, u\ge
  s\}$.
  \begin{multline*}
    \esp{|Y_{r,s}(x)-Y_{r^\prime,s^\prime}(x)|^p}\\
    \begin{aligned}
      &\le c\left( \esp{|Y_{r,s}(x)-Y_{r^\prime,s}(x)|^p}+\esp{|Y_{r^\prime,s}(x)-Y_{r^\prime,s^\prime}(x)|^p}\right)\\
      &=c\left( \esp{|Y_{r,s}(x)-Y_{r^\prime,s}(x)|^p}+\esp{|Y_{r^\prime,s}(x)-Y_{r^\prime,s}(Y_{s,s^\prime}(x))|^p}\right)\\
      &=c\left(
        \esp{|Z_{s-r,s}(x)-Z_{s-r^\prime,s}(x)|^p}+\esp{|Z_{s-r^\prime,s}(x)-Z_{s-r^\prime,s}(Y_{s,s^\prime}(x))|^p}\right).
    \end{aligned}
  \end{multline*}
  According to Theorem \ref{thm:continuite_en_temps_Z},
  \begin{equation}\label{eq:15}
    \esp{|Z_{s-r,s}(x)-Z_{s-r^\prime,s}(x)|^p}\le c|r-r^\prime|^{p\eta}(1+|x|^p).
  \end{equation}
  In view of Theorem \ref{thm:int_strato_inverse}, the stochastic
  integral which appears in Equation (\ref{eq:C}) is also a
  Stratonovitch integral hence we can apply the substitution formula
  and say
  \begin{equation*}
    Z_{s-r^\prime,s}(Y_{s,s^\prime}(x))=\left. Z_{s-r^\prime,s}(y)\right|_{y=Y_{s,s^\prime}(x)}.
  \end{equation*}
  Thus we can apply Theorem \ref{thm:continuite_en_temps_Z} and we
  obtain
  \begin{equation*}
    \esp{|Z_{s-r^\prime,s}(x)-Z_{s-r^\prime,s}(Y_{s,s^\prime}(x))|^p}\le  c\esp{|x-Y_{s,s^\prime}(x)|^p}.
  \end{equation*}
  The right hand side of this equation is in turn equal to
  \begin{math}
    \esp{|Z_{0,s^\prime}-Z_{s^\prime-s,s^\prime}(x)|^p},
  \end{math}
  thus, we get
  \begin{equation}\label{eq:16}
    \esp{|Z_{s-r^\prime,s}(x)-Z_{s-r^\prime,s}(Y_{s,s^\prime}(x))|^p}\le  c  (1+|x|^p) |s^\prime-s|^{p\eta}
  \end{equation}
  Combining (\ref{eq:15}) and (\ref{eq:16}) gives
  \begin{equation*}
    \esp{|Y_{r,s}(x)-Y_{r^\prime,s^\prime}(x)|^p}\le c (1+|x|^p) ( |s^\prime-s|^{p\eta}+|r-r^\prime|^{p\eta}),
  \end{equation*}
  hence the result.
\end{proof}

Thus, we have the main result of this paper.
\begin{theorem}
  \label{thm:solution_A}
  Assume that $\VV$ is an $E^0$ causal map continuous from $\L^p$ into
  $\I_{\alpha,p}$ for $\alpha >0$ and $p\ge 4$ such that $\alpha p>1.$
  Then Equation (\ref{eq:Ap}) has one and only one solution.
\end{theorem}
\begin{proof}
  Under the hypothesis, we know that Equation (\ref{eq:B1}) has a
  unique solution which satisfies (\ref{eq:14}). By definition of a
  solution of (\ref{eq:B1}), the process $Y^{-1}:(\omega,\, s)\mapsto
  Y_{st}^{-1}(\omega,\, x)$ belongs to $\LL_{p,1}$ hence we can apply
  the substitution formula. Following the lines of proof of the
  previous theorem, we see that $Y^{-1}$ is a solution of
  (\ref{eq:Ap}).

  In the reverse direction, two distinct solutions of (\ref{eq:Ap})
  would give raise to two solutions of (\ref{eq:B1}) by the same
  principles. Since this is definitely impossible in view of Theorem
  \ref{thm:solution_B}, Equation (\ref{eq:Ap}) has at most one
  solution.
\end{proof}

\section{Technical proofs}
\label{sec:proofs}

\subsection{Substitution formula}
\label{sec:substitution-formula}
The proof of \ref{thm:int_strato_inverse} relies on several lemmas
including one known in anticipative calculus as the substitution
formula, cf. \cite{nualart.book}.
\begin{theorem}
  \label{thm:existence_trace}
  Assume that Hypothesis \ref{hypA}($p,\, \eta$) holds. Let $u$ belong
  to $\LL_{p,1}.$ If $V\n^\WW u$ is of trace class, then
  \begin{equation*}
    \int_0^T D^\WW u(s)\d s=\trace(V\n^\WW u).
  \end{equation*}
  Moreover,
  \begin{equation*}
    \esp{\left|\trace(V\n^\WW u)\right|^p}\le \, c\, \|u\|_{\LL_{p,1}}^p.
  \end{equation*}
\end{theorem}
\begin{proof}
  For each $k$, let $(\phi_{k,\, m},\, m=1,\cdots,\, 2^k)$ be the
  functions
  \begin{math}
    \phi_{k,\, m}=2^{k/2}\car_{[(m-1)2^{-k},\, m2^{-k})}.
  \end{math}
  Let $P_k$ be the projection onto the span of the $\phi_{k,\, m}$,
  since $\n^\WW Vu$ is of trace class, we have (see \cite{MR2154153})
$$\trace(V\n^\WW  p_t u)=\lim_{k\to +\infty} \trace(P_k \, V\n^\WW  p_t u \, P_k).$$ Now,
\begin{align*}
  \trace(P_k \, V\n^\WW u \, P_k)&=\sum_{m=1}^k (V\n^\WW p_t u,\
  \phi_{k,m}\otimes \phi_{k,m})_{\L^2\otimes \L^2}\\
  &=\sum_{m=1}^k2^k \int_{(m-1)2^{-k}\wedge t}^{ m2^{-k}\wedge t }
  \int_{(m-1)2^{-k}\wedge t }^{ m2^{-k}\wedge t} V(\n^\WW _ru)(s) \d
  s\d r.
\end{align*}
According to the proof of Theorem \ref{thm:strato_integrable}, the
first part of the theorem follows. The second part is then a rewriting
of (\ref{eq:121}).
\end{proof}
For 
$p\ge 1$, let $\Gamma_p$ be the set of random fields:
\begin{align*}
  u\, :\, \R^m &\longrightarrow \LL_{p,1}\\
  x&\longmapsto ((\omega,\, s)\mapsto u(\omega, \, s,\, x))
\end{align*}
equipped with the semi-norms,
\begin{equation*}
  p_K(u)=\sup_{x\in K} \| u(x)\|_{\LL_{p,1}}
\end{equation*}
for any compact $K$ of $\R^m$.
\begin{corollary}[Substitution formula]
  \label{cor:substitution}
  Assume that Hypothesis \ref{hypA}($p,\, \eta$) holds.  Let
  $\{u(x),\,, x\in \R^m\}$ belong to $\Gamma_p$. Let $F$ be a random
  variable such that $((\omega, \, s)\mapsto u(\omega, \, s,\, F))$
  belongs to $\LL_{p,1}$. Then,
  \begin{equation}\label{eq:12}
    \int_0^T u(s,F)\circ \d W^V(s) =\left.\int_0^T u(s,x)\circ \d W^V_s\right|_{x=F}.
  \end{equation}
\end{corollary}
\begin{proof}
  Simple random fields of the form
  \begin{equation*}
    u(\omega,\, s,\, x)=\sum_{l=1}^K H_l(x)u_l(\omega,\, s)
  \end{equation*}
  with $H_l$ smooth and $u_l$ in $\LL_{p,1}$ are dense in
  $\Gamma_p$. In view of (\ref{eq:3}), it is sufficient to prove the
  result for such random fields. By linearity, we can reduce the proof
  to random fields of the form $H(x)u(\omega,\, s).$ Now for any
  partition $\pi$,
  \begin{multline*}
    \div^\WW\Bigl(\sum_{t_i\in\pi}\frac{1}{\theta_i}
    \int_{t_i\wedge t}^{t_{i+1}\wedge t} \kern-3pt H( F)V(u(\omega, .))(r) \d r\, \car_{[t_i,\, t_{i+1})}\Bigr)\\
    =H( F) \div^\WW\Bigl(\sum_{t_i\in\pi}\frac{1}{\theta_i}
    \int_{t_i\wedge t}^{t_{i+1}\wedge t} \kern-3pt V(u(\omega, .))(r) \d r\, \car_{[t_i,\, t_{i+1})}\Bigr)\\
    -\sum_{t_i\in\pi}\int_{t_i\wedge t}^{t_{i+1}\wedge
      t}\int_{t_i\wedge t}^{t_{i+1}\wedge t} H^\prime(F)\n^\WW_s F \,
    Vu(r)\d s\d r.
  \end{multline*}
  On the other hand,
  \begin{align*}
    \n^\WW_s (H(F)u(\omega,\, r))&=H^\prime(F)\n^\WW_s F \, u(r),
  \end{align*}
  hence
  \begin{multline*}
    \sum_{t_i\in\pi}\frac{1}{\theta_i} \iint\limits_{[t_i\wedge
      t,t_{i+1}\wedge t]^2} V(\n^\WW _rH(F)u)(s) \d s\d r\\
    =\sum_{t_i\in\pi}\frac{1}{\theta_i} \iint\limits_{[t_i\wedge
      t,t_{i+1}\wedge t]^2} H^\prime(F)\n^\WW_s F \, Vu(r) \d s\d r.
  \end{multline*}
  According to Theorem \ref{thm:strato_integrable}, Eqn. (\ref{eq:12})
  is satisfied for simple random fields.
\end{proof}
\begin{defn}
  \label{def:integrale}
  For any $0\le r\le t\le T$, for $u$ in $\LL_{p,\, 1}$, we define
  \begin{math}
    \int_r^t u(s)\circ \d W^V(s)
  \end{math}
  as
  \begin{align*}
    \int_r^t u(s)\circ \d W^V(s)&=\int_0^t u(s)\circ \d
    W^V(s)-\int_0^r u(s)\circ \d W^V(s)\\
    &= \int_0^T e_tu(s)\d W^V(s)-\int_0^T e_ru(s)\circ \d W^V(s)\\
    &=\div^\WW(V(e_t-e_r)u)+\int_r^t D^\WW u(s)\d s. \end{align*}
\end{defn}
By the very definition of trace class operators, the next lemma is straightforward.
\begin{lemma}
  \label{lem:taut_trace}
  Let $A$ and $B$ be two continuous maps from $\L^2([0,\, T];\, \R^n)$
  into itself. Then, the map $\taut A\otimes B$ (resp. $A\taut \otimes
  B$) is of trace class if and only if the map $A\otimes \taut B$
  (resp. $A\otimes B \taut$) is of trace class. Moreover, in such a
  situation,
  \begin{equation*}
    \trace(\taut A\otimes B)=\trace( A\otimes \taut B) \text{, resp. } \trace( A\taut \otimes B)=\trace( A\otimes  B\taut).
  \end{equation*}
\end{lemma}
The next corollary follows by a classical density argument.
\begin{corollary}
  \label{lem:taut}
  Let $u\in \LL_{2,1}$ such that $\n^\WW\otimes \taut V u$ and
  $\n^\WW\otimes V \taut u$ are of trace class. Then, $\taut
  \n^\WW\otimes V u$ and $\n^\WW \taut \otimes V u $ are of trace
  class. Moreover, we have:
  \begin{multline*}
    \trace(\n^\WW\otimes \taut V u)=\trace(\taut \n^\WW\otimes  V u)\\
    \text{ and } \trace(\n^\WW\otimes (V \taut) u)=\trace( \n^\WW
    \taut \otimes V u).
  \end{multline*}
\end{corollary}

\begin{proof}[Proof of~\ref{thm:int_strato_inverse}]
  We first study the divergence term. In view of \ref{thm:taut_delta},
  we have
  \begin{align*}
    \div^B (V(e_{T-r}-e_{T-t}) \taut \cu \circ \Theta_T)&=\div^B (V\taut (e_t-e_r) \cu \circ \Theta_T)\\
    &= \div^B (\taut \VV (e_t-e_r) \cu \circ \Theta_T)\\
    &=\cdelta(\VV (e_t-e_r) \cu)(\comega)\\
    &=\int_r^t \VV (\car_{[r,\, t]}\cu)(s)\d B^T(s).
  \end{align*}
  According to Theorem \ref{lem:causalite}, $(\VV)^*$ is $\check{E}_0$
  causal and according to \ref{lem:strict-causality}, it is strictly
  $\check{E}_0$ causal. Thus, Theorem \ref{thm:trace_zero} implies
  that $\cn V (e_t-e_r) \cu$ is of trace class and
  quasi-nilpotent. Hence Lemma \ref{lem:taut} induces that
  \begin{equation*}
    \taut \VV \taut \otimes \taut \cn\taut (e_t-e_r) \cu
  \end{equation*}
  is trace-class and quasi-nilpotent. Now, according to Theorem
  \ref{thm:cdelta}, we have
  \begin{equation*}
    \taut \VV \taut \otimes \taut \cn\taut (e_t-e_r) \cu=V(\n \taut (e_{T-r}-e_{T-t})\cu \circ \Theta_T).
  \end{equation*}
  According to Theorem \ref{thm:strato_integrable}, we have proved
  (\ref{eq:9}).
\end{proof}
\subsection{The forward equation}
\label{sec:equation-refeq:c}
\begin{lemma}
  \label{lem:borne_VVsigma}
  Assume that Hypothesis \ref{hypA}$(p,\, \eta)$ holds and that
  $\sigma$ is Lipschitz continuous. Then, for any $0\le a\le b\le T$,
  the map
  \begin{align*}
    \VV\circ \sigma \, :\, C([0,T], \, \R^n)&\longrightarrow C([0,T], \, \R^n)\\
    \phi &\longmapsto \VV(\sigma\circ \psi\ \car_{[a,b]})
  \end{align*}
  is Lipschitz continuous and Gâteaux differentiable. Its differential
  is given by:
  \begin{equation}\label{eq:18}
    d\VV\circ\sigma(\phi)[\psi]=\VV(\sigma^\prime\circ \phi \ \psi).
  \end{equation}
  Assume furthermore that $\sigma$ is sub-linear, i.e.,
  \begin{equation}\label{eq:22}
    |\sigma(x)|\le c ( 1+|x|), \text{ for any }x\in \R^n.
  \end{equation}
  Then, for any $\psi\in C([0,T],\, \R^n)$, for any $t\in [0,T]$,
  \begin{align*}
    |\VV(\sigma\circ \psi)(t)| &\le c T^{\eta+1/p} (1+\int_0^t |\psi(s)|^p \d s)\\
    &\le c T^{\eta+1/p} (1+\| \psi\|_\infty).
  \end{align*}
\end{lemma}
\begin{proof}
  Let $\psi$ and $\phi$ be two continuous functions, since $ C([0,T],
  \, \R^n)$ is continuously embedded in $\L^p$,
  $\VV(\sigma\circ\psi-\sigma\circ \phi)$ belongs to
  $\Hol(\eta)$. Moreover,
  \begin{align*}
    \sup_{t\le T} | \VV(\sigma\circ \psi\
    \car_{[a,b]})(t)-\VV(\sigma\circ \phi\ \car_{[a,b]})(t) | & \le
    c\,\|\VV((\sigma\circ\psi-\sigma\circ \phi)  \ \car_{[a,b]})\|_{\Hol(\eta)}\\
    &\le c \,\|(\sigma\circ\psi-\sigma\circ \phi)\ \car_{[a,b]}\|_{\L^p}\\
    &\le c \,\|\phi-\psi\|_{\L^p([a,\, b])}\\
    &\le c \,\sup_{t\le T} | \psi(t)-\phi(t) |,
  \end{align*}
  since $\sigma$ is Lipschitz continuous.

  Let $\psi$ and $\psi$ two continuous functions on $[0,T]$. Since
  $\sigma$ is Lipschitz continuous, we have
  \begin{equation*}
    \sigma(\psi(t)+\ee \phi(t))=\sigma(\psi(t))+\ee\int_0^1 \sigma^\prime(u\psi(t)+(1-u)\phi(t))\d u.
  \end{equation*}
  Moreover, since $\sigma$ is Lipschitz, $\sigma^\prime$ is bounded
  and
  \begin{equation*}
    \int_0^T \left| \int_0^1 \sigma^\prime(u\psi(t)+(1-u)\phi(t))\d
      u\right|^p \d t \le c \, T.
  \end{equation*}
  This means that $(t\mapsto \int_0^1
  \sigma^\prime(u\psi(t)+(1-u)\phi(t))\d u)$ belongs to $\L^p.$ Hence,
  according to Hypothesis \ref{hypA},
  \begin{equation*}
    \|  \VV(\int_0^1 \sigma^\prime(u\psi(.)+(1-u)\phi(.))\d u)\|_C\le c T.
  \end{equation*}
  Thus,
  \begin{equation*}
    \lim_{\ee\to 0} \ee^{-1}(\VV(\sigma\circ (\psi+\ee\phi))-\VV(\sigma\circ \psi)) \text{ exists,}
  \end{equation*}
  and $\VV\circ \sigma$ is Gâteaux differentiable and its differential
  is given by \eqref{eq:18}.

  Since $\sigma\circ\psi$ belongs to $C([0,T],\, \R^n)$, according to
  Hypothesis \ref{hypA}, we have:
  \begin{align*}
    |\VV(\sigma\circ \psi)(t)|&\le c\left(\int_0^t s^{\eta p}|\sigma(\psi(s))|^p \d s\right)^{1/p}\\
    &\le cT^\eta  \left(\int_0^t (1+|\psi(s)|^p) \d s\right)^{1/p}\\
    &\le c T^{\eta+1/p} (1+\|\psi\|_\infty^p )^{1/p}\\
    &\le c T^{\eta+1/p} (1+\|\psi\|_\infty).
  \end{align*}
  The proof is thus complete.
\end{proof}
Following \cite{MR658690}, we then have the following non trivial
result.
\begin{theorem}
  \label{thm:continuite_en_temps_Z}
  Assume that Hypothesis \ref{hypA}$(p,\, \eta)$ holds and that
  $\sigma$ is Lipschitz continuous.  Then, there exists one and only
  one measurable map from $\Omega \times [0,T] \times [0,T]$ into $\G$
  which satisfies the first two points of Definition
  \eqref{eq:C}. Moreover,
  \begin{equation*}
    \esp{|Z_{r,\, t}(x)-Z_{r^\prime,\, t} (x^\prime)|^p}\le c
    (1+|x|^p\vee |x^\prime|^p)\, \left(|r-r^\prime|^{p\eta}+|x-x^\prime|^p\right)
  \end{equation*}
and 
 for any $x\in \R^n$, for any $0\le r\le t \le
  T,$ we have
  \begin{equation*}
    \esp{|Z_{r,t}(x)|^p}\le c (1+|x|^p)e^{c T^{\eta p +1}}.
  \end{equation*}
  Note even if $x$ and $x^\prime$ are replaced by
  $\sigma\{\check{B}^T(u),\, t\le u\}$ measurable random variables,
  the last estimates still holds.
\end{theorem}
\begin{proof}
  Existence, uniqueness and homeomorphy of a solution of \eqref{eq:C}
  follow from \cite{MR658690}. The regularity with respect to $r$ and
  $x$ is obtained as usual by BDG inequality and Gronwall Lemma. For
  $x$ or $x^\prime$ random, use the independence of
  $\sigma\{\check{B}^T(u),\, t\le u\}$ and $\sigma\{\check{B}^T(u),\,
  r\wedge r^\prime\le u \le t \}$.
\end{proof}

\begin{theorem}
  Assume that Hypothesis \ref{hypA}$(p,\, \eta)$ holds and that $\sigma$ is
  Lipschitz continuous and sub-linear. Then, for any $x\in \R^n$, for
  any $0\le r\le s\le t\le T$, $(\omega,\, r)\mapsto Z_{r,s}(\omega,
  \, Z_{s,t}(x))$ and $(\omega,\, r)\mapsto Z_{r,t}^{-1}(\omega, \,
  x)$ belong to $\LL_{p,1}.$
\end{theorem}
\begin{proof}
  According to \cite[Theorem 3.1]{hirsch88}, the differentiability of
  $\omega\mapsto Z_{r,t}(\omega, \, x)$ is ensured. Furthermore,
  \begin{equation*}
    \nabla_u Z_{r,t}=- \VV(\sigma\circ Z_{.,t}\car_{[r,\, t]})(u)-\int_r^t  \VV
    (\sigma^\prime(Z_{.,t} ).\nabla_u Z_{.,t}\car_{[r,\, t]})(s) \d \check{B}(s),
  \end{equation*}
  where $\sigma^\prime$ is the differential of $\sigma.$ For $M>0$, let 
  \begin{equation*}
    \xi_M=\inf\{\tau,\, |\nabla_u Z_{\tau,\, t}|^p\ge M\} \text{ and }
    Z^M_{\tau,\ t}=Z_{\tau\vee \xi_M,\, t}.
  \end{equation*}
Since $\VV$
  is continuous from $\L^p$ in to itself and $\sigma$ is Lipschitz,
  according to BDG inequality, for $r\le u $,
  \begin{multline*}
    \esp{ |\nabla_u Z^M_{r,t}|^p }\\
    \begin{aligned}
      & \le c \esp{ |\VV(\sigma\circ Z^M_{.,t}\car_{[r,\, t]})(u)|^p }
      +c \, \esp{ \int_r^t |\VV(\sigma^\prime( Z^M_{.,t} )\, \nabla_u
        Z^M_{., t}\car_{[r,\, t]})(s)|^p \d s}\\
      &\le c \left(1+\esp{\int_r^t u^{p\eta} \int_r^u |Z_{\tau,t}|^p
          \d \tau\d u}
        + \esp{\int_r^t s^{p\eta} \int_r^s |\nabla_u Z^M_{\tau ,t}|^p\d\tau\ \d s}  \right)\\
      &\le c \left(1+\esp{\int_r^t |Z_{\tau,t}|^p
          (t^{p\eta+1}-\tau^{p\eta+1})\d \tau}
        + \esp{\int_r^t  |\nabla_u Z^M_{\tau ,t}|^p  (t^{p\eta+1}-\tau^{p\eta+1})\d\tau}\right)\\
      &\le ct^{p\eta+1} \left(1+\esp{\int_r^t |Z_{\tau,t}|^p \d \tau}
        + \esp{\int_r^t |\nabla_u Z^M_{\tau ,t}|^p \d\tau}\right).
    \end{aligned}
  \end{multline*}
  Then, Gronwall Lemma entails that
  \begin{equation*}
    \esp{|\nabla_u Z^M_{r,t}|^p }\le  c\, \left(1+\esp{\int_r^t |Z_{\tau,t}|^p \d \tau}\right),
  \end{equation*}
hence by Fatou lemma,
    \begin{equation*}
    \esp{|\nabla_u Z_{r,t}|^p }\le  c\, \left(1+\esp{\int_r^t |Z_{\tau,t}|^p \d \tau}\right).
  \end{equation*}
  The integrability of $\esp{|\nabla_u Z_{r,t}|^p }$ with respect to
  $u$ follows.

  Now, since $0\le r\le s\le t\le T$, $Z_{s,t}(x)$ is independent of
  $Z_{r,s}(x),$ thus the previous computations still hold and
  $(\omega,\, r)\mapsto Z_{r,s}(\omega, \, Z_{s,t}(x))$ belong to
  $\LL_{p,1}$.

  According to \cite{MR810975}, to prove that $Z_{r,t}^{-1}(x)$
  belongs to $\D_{p,1}$, we need to prove
  \begin{enumerate}
  \item for every $h\in \L^2$, there exists an absolutely continuous
    version of the process $(t\mapsto Z_{r,t}^{-1}(\omega+th,\, x))$,
  \item there exists $D Z_{r,t}^{-1}$, an $\L^2$-valued random
    variable such that for every $h\in \L^2$,
    \begin{equation*}
      \frac{1}{t}(    Z_{r,t}^{-1}(\omega+th,\, x)-Z_{r,t}^{-1}(\omega,\,
      x))\xrightarrow{t\to 0} \int_0^T D Z_{r,t}^{-1}(s) h(s)\d s,
    \end{equation*}
    where the convergence holds in probability,
  \item $D Z_{r,t}^{-1}$ belongs to $\L^2(\Omega,\, \L^2)$.
  \end{enumerate}
  We first show that
  \begin{equation}\label{eq:19}
    \esp{  \left| \frac{\partial
          Z_{r,t}}{\partial x}(\omega,\, Z_{r,t}^{-1}(x))\right|^{-p} }\text{ is finite.}
  \end{equation}
  Since
  \begin{equation*}
    \frac{\partial      Z_{r,t}}{\partial x}(\omega,\, x)=\id + \int_r^t
    \VV(\sigma^\prime(Z_{.,t}(x))  \frac{\partial     Z_{.,t}(\omega,\, x)}{\partial
      x})(s)\d \check{B}(s),
  \end{equation*}
  Let $\Theta_v=\sup_{u\le v} |\partial_x Z_{u,\, t}(x)|$. The same
  kind of computations as above entails that (for the sake of brevity,
  we do not detail the localisation procedure as it is similar to the
  previous one):
  \begin{multline*}
    \esp{\Theta_v^{2q}} \le c+c\, \esp{\int_u^t \Theta_s^{2(q-1)}
      \left(\int_u^s
        |\partial_x Z_{\tau,\, t}(x)|^{p}|\d\tau\right)^{2/p}\d s}\\
    +c\, \esp{\left(\int_u^t \Theta_s^{q-2} \left(\int_u^s |\partial_x
          Z_{\tau,\, t}(x)|^{p}|\d\tau\right)^{2/p}\right)^2\d s}.
  \end{multline*}
  Hence,
  \begin{equation*}
    \esp{\Theta_v^{2q}} 
    \le c\left(1+\int_v^t \esp{\Theta_s^{2q}}\d s\right),
  \end{equation*}
  and \eqref{eq:19} follows by Fatou and Gronwall lemmas.  Since $Z_{r,t}(\omega,
  \, Z_{r,t}^{-1}(\omega, \, x))=x,$ the implicit function theorem
  imply that $Z_{r,\, t}^{-1}(x)$ satisfies the first two properties
  and that
  \begin{equation*}
    {\nabla}Z_{r,t}(\omega,\, Z_{r,t}^{-1}(x))+\frac{\partial
      Z_{r,t}}{\partial x}(\omega,\, Z_{r,t}^{-1}(x))\tilde{\nabla}Z^{-1}_{r,t}(\omega,\, x).
  \end{equation*}
  It follows by Hölder inequality and Equation \eqref{eq:19} that
  \begin{equation*}
    \|  DZ_{r,t}^{-1}(x))\|_{p,1}\le c \|
    Z_{r,t}(x))\|_{2p,1}\|(\partial_x Z_{r,\, t}(x))^{-1}\|_{2p},
  \end{equation*}
  hence $Z_{r,\, t}^{-1}$ belongs to $\LL_{p,\, 1}.$
\end{proof}
\appendix

\end{document}